\documentclass[reqno]{amsart}
\usepackage{hyperref}
\usepackage{amssymb}
\usepackage{amsmath}
\usepackage{colortbl}
\usepackage[pagewise]{lineno}

\begin{document}
\title[\hfilneg \quad \quad \quad \quad \quad \quad \quad \quad \quad \quad \quad \quad \quad \quad Anisotropic p-Laplace equation]
{Existence and Nonexistence results for anisotropic p-Laplace equation with singular nonlinearities}
\author[P. Garain]
{Prashanta Garain}

\address{Prashanta Garain \newline
Aalto University, Department of Mathematics and Systems Analysis, Otakaari 1, 02150 Espoo, Finland}
\email{pgarain92@gmail.com}

\keywords{Anisotropic p-Laplacian; Existence; Nonexistence; Stable solution 
\hfill\break\indent}

\begin{abstract}
Let $p_i\geq 2$ and consider the following anisotropic $p$-Laplace equation
$$
-\sum_{i=1}^{N}\frac{\partial}{\partial x_i}\Big(\Big|\frac{\partial u}{\partial x_i}\Big|^{p_i-2}\frac{\partial u}{\partial x_i}\Big)=g(x)f(u),\,\,u>0\text{ in }\Omega.
$$
Under suitable hypothesis on the weight function $g$ we present an existence result for $f(u)=e^\frac{1}{u}$ in a bounded smooth domain $\Omega$ and nonexistence results for $f(u)=-e^\frac{1}{u}$ or $-(u^{-\delta}+u^{-\gamma}),$ $\delta,\gamma>0$ with $\Omega=\mathbb{R}^N$ respectively.
\end{abstract}       

\maketitle\numberwithin{equation}{section}
\newtheorem{theorem}{Theorem}[section]
\newtheorem{lemma}[theorem]{Lemma}
\newtheorem{Def}[theorem]{Definition}
\newtheorem{remark}[theorem]{Remark}
\newtheorem{Pro}{Proof}[section]
\newtheorem{corollary}[theorem]{Corollary}
\allowdisplaybreaks

\section{Introduction}
In this article we are interested in the question of existence of a weak solution to the following anisotropic $p$-Laplace equation
\begin{equation}\label{weaksoln}
-\sum_{i=1}^{N}\frac{\partial}{\partial x_i}\Big(\Big|\frac{\partial u}{\partial x_i}\Big|^{p_i-2}\frac{\partial u}{\partial x_i}\Big)=g(x)e^\frac{1}{u},\,\,u>0\text{ in }\Omega
\end{equation}
where $\Omega$ is a bounded smooth domain in $\mathbb{R}^N$ with $N\geq 3$ and $g\in L^1(\Omega)$ is nonnegative which is not identically zero.

Alongside we present nonexistence results concerning stable solutions to the following equation
\begin{equation}\label{main equation}
\begin{gathered}
-\sum_{i=1}^{N}\frac{\partial}{\partial x_i}\Big(\Big|\frac{\partial u}{\partial x_i}\Big|^{p_i-2}\frac{\partial u}{\partial x_i}\Big)=g(x)f(u)\;\;\mbox{in}\;\;\mathbb{R}^N,\,u>0\,\text{ in }\mathbb{R}^N
\end{gathered}
\end{equation}
where $f(u)$ is either $-(u^{-\delta}+u^{-\gamma})$ with $\delta,\gamma>0$ or $-e^\frac{1}{u}.$ The weight function $g\in L^1_{loc}(\mathbb{R}^N)$ is such that $g\geq c>0$ for some constant $c$.

Throughout the article, we assume that $p_i\geq 2.$ If $p_i=2$ for all $i$ and $g\equiv 1$ equation \eqref{main equation} becomes the Laplace equation
\begin{equation}\label{Laplace}
-\Delta u=f(u)\text{ in }\Omega. 
\end{equation} 
Observe that the nonlinearities in our consideration is singular in the sense that it blows up near the origin. Starting from the pioneering work of Crandall et al \cite{CRT} where the existence of a unique positive classical solution for $f(u)=u^{-\delta}$ with any $\delta>0$ has been proved for the problem \eqref{Laplace} with zero Dirichlet boundary value.
Lazer-Mckenna \cite{Lazer} observed that the above classical solution is a weak solution in $H_0^{1}(\Omega)$ iff $0<\delta<3.$ Boccardo-Orsina \cite{Boc} investigated the case of any $\delta>0$ concerning the existence of a weak solution in $H_{loc}^1(\Omega)$. Moreover existence and uniqueness of solutions for singular Laplace equation has been investigated by Canino-Degiovanni \cite{Caninonew}, Canino-Sciunzi  \cite{Caninonew1}. Canino et al \cite{Canino1} generalized the problem \eqref{Laplace} to the following singular $p$-Laplace equation
\begin{equation}\label{Singlap}
-\Delta_{p}u=\frac{f(x)}{u^\delta}\text{ in }\Omega,\,\,u>0\text{ in }\Omega,\,\,u=0\text{ on }\partial\Omega
\end{equation}
to obtain existence and uniqueness of weak solution for any $\delta>0$ under suitable hypothesis on $f$. For more details concerning singular problems reader can look at \cite{ARCOYA, BalJM, Radu} and the references therein.

Farina \cite{Farina1} settled the question of nonexistence of stable solution for the equation \eqref{Laplace} with $f(u)=e^u.$ There is a huge literature in this direction for various type of nonlinearity $f(u),$ reader can look at the nice surveys \cite{Farina2, Dupaigne}. For $f(u)=-u^{-\delta}$ with $\delta>0$ Ma-Wei \cite{jwei} proved that the equation \eqref{Laplace} does not admit any $C^1(\mathbb{R}^N)$ stable solution provided
$$
2\leq N<2+\frac{4}{1+\delta}\left(\delta+\sqrt{\delta^2+\delta}\right).
$$
Moreover many other qualitative properties of solutions has been obtained there.
Consider the weighted $p$-Laplace equation
\begin{equation}\label{plap}
-\text{div}\left(w(x)|\nabla u|^{p-2}\nabla u\right)=g(x)f(u)\text{ in }\mathbb{R}^N.
\end{equation}
For $w=g=1,$ Guo-Mei \cite{GuoMei} showed nonexistence results in $C^1(\mathbb{R}^N)$ for \eqref{plap}, provided $2\leq p<N<\frac{p(p+3)}{p-1}$ and $\delta>q_c$ where
$$
q_c=\frac{(p-1)[(1-p)N^2+(p^2+2p)N-p^2]-2p^2\sqrt{(p-1)(N-1)}}{(N-p)[(p-1)N-p(p+3)]}.
$$
By considering a more general weight $g\in L^1_{loc}(\mathbb{R}^N)$ such that $|g|\geq C|x|^a$ for large $|x|,$ Chen et al \cite{Chen} proved nonexistence results for the equation \eqref{plap}, provided $w=1$ and $2\leq p<N<\frac{p(p+3)+4a}{p-1}$ and $\delta>q_c$ where
$$
q_c=\frac{2(N+a)(p+a)-(N-p)[(p-1)(N+a)-p-a]-\beta}{(N-p)[(p-1)N-p(p+3)]},
$$ 
for
$$
\beta=2(p+a)\sqrt{(p+a)\Big(N+a+\frac{N-p}{p-1}\Big)}.
$$
Recently this has been extended for a general weight function $w$ in \cite{Garain2,Ple}. 

Our main motive in this article is to investigate such results in the framework of the anisotropic $p$-Laplace operator, which is non-homogeneous. Such operators appear in many physical phenomena, for example,  it reflects anisotropic physical properties of some reinforced materials \cite{Lions}, appears in image processing \cite{image}, to study the dynamics of fluids in anisotropic media when the conductivities of the media are different in each direction \cite{Diaz}. The first part of this article is devoted to the existence of a weak solution for the anisotropic problem \eqref{weaksoln}. Some recent works on singular anisotropic problems can be found in \cite{Aniso1, Aniso2}. The singularity $e^\frac{1}{u}$ is more singular in nature compared to $u^{-\delta}$ which protects one to obtain the uniform boundedness of $u_n$ as in \cite{Boc}. We overcome this difficulty using the domain approximation method following \cite{Perera}. In the second part we provide nonexistence results of stable solutions for the anisotropic $p$-Laplace equation \eqref{main equation} with the mixed singularities $-(u^{-\delta}+u^{-\gamma})$ and $-e^\frac{1}{u}.$ We employ the idea introduced in \cite{Farina1} to establish our main results stated in section 2 for which Caccioppoli type estimates (see section 5) will be the main ingredient. The main difficulty to obtain such estimates arises due to the nonhomogenity of the anisotropic $p$-Laplace operator which we overcome by choosing suitable test functions in the stability condition.

\section{Preliminaries}
In this section, we present some basic results in the anisotropic Sobolev space.

\textbf{Anisotropic Sobolev Space:} Let $p_i\geq 2$ for all $i$, then for any domain $D$ define the anisotropic Sobolev space by
$$
W^{1,p_i}(D)=\left\{v\in W^{1,1}(D):\frac{\partial v}{\partial x_i}\in L^{p_i}(D)\right\}
$$
and 
$$
W_{0}^{1,p_i}(D)=W^{1,p_i}(D)\cap W_{0}^{1,1}(D)
$$
endowed with the norm
$$
||v||_{W_{0}^{1,p_i}(D)}=\sum_{i=1}^{N}\Big|\Big|\frac{\partial v}{\partial x_i}\Big|\Big|_{L^{p_i}(D)}.
$$
The space $W^{1,p_i}_{loc}(D)$ is defined analogously. 

We denote by $\bar{p}$ to be the harmonic mean of $p_1,p_2,\cdots,p_N$ defined by
$$
\frac{1}{\bar{p}}=\frac{1}{N}\sum_{i=1}^{N}\frac{1}{p_i}
$$
and
$$
\bar{p}^{*}=\frac{N\bar{p}}{N-\bar{p}}.
$$
The following Sobolev embedding theorem can be found in \cite{castrothesis, Castro, Anisoemb}.
\begin{theorem}\label{Sobolevemb}
For any bounded domain $\Omega$, the inclusion map
$$
W_{0}^{1,p_i}(\Omega)\to L^r(\Omega)
$$
is continuous for every $r\in[1,\bar{p}^{*}]$ if $\bar{p}<N$ and for every $r\geq 1$ if $\bar{p}\geq N.$ 
Moreover, there exists a positive constant $C$ depending only on $\Omega$ such that for every $v\in W_{0}^{1,p_i}(\Omega)$
$$
||v||_{L^r(\Omega)}\leq C\prod_{i=1}^{N}||\frac{\partial v}{\partial x_i}||_{L^{p_i}(\Omega)},\,\,\forall\,\,r\in[1,\bar{p}^{*}].
$$
\end{theorem}

\textbf{Weak Solution:} We say that $u\in W^{1,p_i}_{loc}(\Omega)$ is a weak solution of the problem \eqref{weaksoln} if $u>0$ a.e. in $\Omega$ and for every $\phi\in C_c^{1}(\Omega)$ 
\begin{equation}\label{wkform}
\sum_{i=1}^{N}\int_{\Omega}\Big|\frac{\partial u}{\partial x_i}\Big|^{p_i-2}\frac{\partial u}{\partial x_i}\frac{\partial\varphi}{\partial x_i} \,dx=\int_{\Omega}g(x)e^\frac{1}{u}\phi\,dx.
\end{equation}

\textbf{Stable Solution:} We say that $u\in W_{loc}^{1,p_i}(\mathbb{R}^N)$ is a stable solution of the problem $(\ref{main equation})$, if $u>0$ a.e. in $\Omega$ such that both $g(x)f(u), g(x)f'(u)\in L^1_{loc}(\mathbb{R}^N)$ and for all $\varphi\in C_{c}^{1}(\mathbb{R}^N)$,
\begin{equation}\label{weak solution}
\sum_{i=1}^{N}\int_{\mathbb{R}^N}\Big|\frac{\partial u}{\partial x_i}\Big|^{p_i-2}\frac{\partial u}{\partial x_i}\frac{\partial\varphi}{\partial x_i} \,dx=\int_{\mathbb{R}^N}g(x)f(u)\varphi \,dx
\end{equation}
and
\begin{equation}\label{stable solution}
\int_{\mathbb{R}^N}g(x)f'(u)\varphi^{2} \,dx\leq\sum_{i=1}^{N}(p_i-1)\int_{\mathbb{R}^N}\Big|\frac{\partial u}{\partial x_i}\Big|^{p_i-2}\Big|\frac{\partial\varphi}{\partial x_i}\Big|^2 \,dx.
\end{equation}

For a general theory of anisotropic Sobolev space, we refer the reader to \cite{castrothesis, Castro, Elsaid, aniso3}.

\subsection*{Assumption and notation for the nonexistence results:}
We denote by $\Omega=\mathbb{R}^N$ for $N\geq 1$ and assume $2<p_1\leq p_2\leq\cdots\leq p_N.$

We will make use of the following truncated functions later. For $k\in\mathbb{N}$, $\alpha>p_{N}-1$ and $t\geq 0$, define
\begin{equation}\nonumber
    a_{k}(t)=
    \begin{cases}
      \frac{(1-\alpha)}{2}{k^\frac{\alpha+1}{2}}\left(t+\frac{1+\alpha}{k(1-\alpha)}\right), & \text{if}\ 0\leq t<\frac{1}{k}, \\\\
      
      t^\frac{1-\alpha}{2}, & \text{if}\ {t\geq \frac{1}{k}},
    \end{cases}
\end{equation}
and
\begin{equation}\nonumber
b_{k}(t)=
\begin{cases}
{-\alpha}{k^{\alpha+1}}\left(t-\frac{1+\alpha}{k\alpha}\right), & \text{if}\ 0\leq t<\frac{1}{k},\\\\
t^{-\alpha}, & \text{if}\ t\geq \frac{1}{k}.
\end{cases}
\end{equation}

Then it can be easily verified that both $a_k$ and $b_k$ are positive $C^1[0,\infty)$ decreasing functions. Moreover, $a_k$ and $b_k$ satisfies the following properties:
\begin{itemize}
\item[a.] $$a_k(t)^2\geq t\,b_k(t),\,\,\forall\,\,t\geq 0.$$ 
\item[b.] $$a_k(t)^{p_i}|a_{k}'(t)|^{2-p_i}+b_k(t)^{p_i}|b_{k}'(t)|^{1-p_i}\leq C\,|t|^{p_i-\alpha-1},$$ for some positive constant $C(p_1,p_2,\cdots,p_N,\alpha)$.
\item[c.] 
$$
a_k'(t)^2=\frac{(\alpha-1)^2}{4\alpha}|b_k'(t)|,\,\,\forall\,\,t\geq 0.
$$
\end{itemize}
The following notations will be used for the nonexistence results.

\textbf{Notation:}\quad
The equation (\ref{main equation}) will be denoted by $(\ref{main equation})_s$ and $(\ref{main equation})_e$ for $f(u)=-u^{-\delta}-u^{-\gamma}$ and $f(u)=-e^\frac{1}{u}$ respectively. Without loss of generality we assume $0<\delta\leq\gamma.$

We denote by $B_{r}(0)$ to be the ball centered at $0$ with radius $r>0$.

We denote by $u_i=\frac{\partial u}{\partial x_i}$ for all $i=1,2,\cdots,N$ and $q=\frac{\sum_{i=1}^{N}p_i}{N}$. 

Denote by $$
l_1=\frac{p_N-q}{2},\,\,l_2=\frac{2\delta}{N(q-1)}-\frac{q-1}{2}\text{ and }l_3=\frac{2}{MN(q-1)}-\frac{q-1}{2}.$$

We denote by $$A=\left(\frac{N(q-1)(p_N-1)}{4},\infty\right),$$
$$
B=\left(0,\frac{4}{N(q-1)(p_N-1)}\right),\,\,C=\left(0,\frac{4}{N(N-1)(q-1)}\right).
$$
Define
$$I=\cap_{i=1}^{N}I_i$$ where
$$
I_i=\left(\frac{N^2(q-1)(p_i-1)}{p_i\big({N(q-1)+4}\big)-N^2(q-1)},\infty\right),
$$
provided $p_i\big({N(q-1)+4}\big)-N^2(q-1)>0$ for all $i=1,2,\cdots,N$ and $J=B\cap C.$

We assume $\delta\in A$ and $M\in J.$ Observe that $\delta\in A$ implies $l_2>l_1$ and $l_2>0.$ Also $M\in J$ implies $l_3>l_1$ and $l_3>0.$

If $C$ depends on $\epsilon$ we denote by $C_\epsilon$ and if $C$ depends on $r_1,r_2,\cdots,r_m$ we denote it by $C(r_1,r_2,\cdots,r_m).$

Throughout this article  $\psi_{R}\in C_{c}^1(\mathbb{R}^N)$ is a  test function such that
\begin{gather*}
0\leq\psi_{R}\leq 1\text{ in }\mathbb{R}^N,\quad
\psi_{R}=1\text{ in } B_{R}(0), \\
\psi_{R}=0\text{ in }\mathbb{R}^N\setminus B_{2R}(0)
\end{gather*}
with
$$
|\nabla\psi_R|\leq\frac{C}{R}
$$
for some constant $C>0$ (independent of $R$).
\section{Main results}
The main results of this article reads as follows:
\begin{theorem}\label{Mainexi}
Let $\Omega\subset\mathbb{R}^N$ be a bounded smooth domain, $N\geq 3$ and $p_N\geq\cdots p_2\geq p_1\geq 2$. Then the problem \eqref{weaksoln} admits a weak solution $u$ in $W^{1,p_i}_{loc}(\Omega)\cap L^\infty(\Omega)$ such that $(u-\epsilon)^{+}\in W_{0}^{1,p_i}(\Omega)$ for every $\epsilon>0,$ provided
\begin{enumerate}
\item[(a)] $g\in L^m(\Omega)$ for some $m>\frac{\bar{p}^{*}}{\bar{p}^{*}-\bar{p}}$ if $\bar{p}<N$ where $\bar{p}^{*}\geq p_N.$ 
\item[(b)] $g\in L^m(\Omega)$ for some $m>\frac{r}{r-p_N}$ if $\bar{p}\geq N$ where $r>p_N.$
\end{enumerate}
\end{theorem}
\begin{theorem}\label{smain1}
Let $u\in W_{loc}^{1,p_i}(\Omega)$ be such that $0<u\leq 1$ a.e. in $\Omega.$ Assume that $1\leq\delta<\gamma$ be such that $\delta\in A\cap I.$ Then $u$ is not a stable solution to the problem $(\ref{main equation})_s.$
\end{theorem}

\begin{theorem}\label{smain2}
Let $u\in W_{loc}^{1,p_i}(\Omega)$ be such that $u\geq 1$ a.e. in $\Omega.$ Assume that $0<\delta<\gamma$ be such that $\delta\in A$ and $\gamma\in I\cap[1,\infty).$  Then $u$ is not a stable solution to the problem $(\ref{main equation})_s.$
\end{theorem}

\begin{theorem}\label{smain3}
Let $u\in W_{loc}^{1,p_i}(\Omega)$ be such that $u>0$ a.e. in $\Omega.$ Assume that $1\leq\delta=\gamma\in A\cap I.$ Then $u$ is not a stable solution to the problem $(\ref{main equation})_s.$
\end{theorem}
\begin{theorem}\label{emain}
Let $u\in W_{loc}^{1,p_i}(\Omega)$ be such that $0<u\leq M$ a.e in $\Omega$, provided $M\in J$. Then $u$ is not a stable solution to the equation $(\ref{main equation})_e$. 
\end{theorem}
We present the proof of the above theorems in the following two sections.
\section{Proof of existence results}
For $n\in\mathbb{N}$, define $g_n(x)=\text{min}\{g(x),n\}$ and consider the following approximnated problem
\begin{equation}\label{approx}
-\sum_{i=1}^{N}\frac{\partial}{\partial x_i}(|u_i|^{p_i-2}u_i)=g_n(x)e^\frac{1}{(u+\frac{1}{n})}\text{ in }\Omega.
\end{equation}
\begin{lemma}\label{approxexi}
Let \begin{enumerate}
\item $g\in L^m(\Omega)$ for some $m>\frac{\bar{p}^{*}}{\bar{p}^{*}-\bar{p}}$ if $\bar{p}<N$ or
\item $g\in L^m(\Omega)$ for some $m>\frac{r}{r-p_N}$ if $\bar{p}\geq N$ where $r>p_N.$
\end{enumerate}
Then for every $n\in\mathbb{N}$ the problem \eqref{approx} has a positive solution $u_n\in W_{0}^{1,p_i}(\Omega).$ Moreover, one has
\begin{itemize}
\item[(i)] $||u_n||_{L^\infty(\Omega)}\leq C$ for some constant $C$ independent of $n$.
\item[(ii)] $u_{n+1}\geq u_n$ and each $u_n$ is unique.
\item[(iii)] there exists a positive constant $c_\omega>0$ such that for every $\omega\subset\subset\Omega$ we have $u_n\geq c_{\omega}>0.$
\end{itemize}
\end{lemma}
\begin{proof}
\textbf{Existence:} Let $v\in L^r(\Omega)$ for some $r\geq 1.$ Then the problem
$$
-\sum_{i=1}^{N}\frac{\partial}{\partial x_i}(|u_i|^{p_i-2}u_i)=g_n(x)e^\frac{1}{(|v|+\frac{1}{n})}
$$
has a unique solution $u=A(v)\in W_{0}^{1,p_i}(\Omega)$ since the r.h.s belongs to $L^\infty(\Omega)$, see \cite{Castro}. Choosing $u=A(v)$ as a test function and using Theorem \ref{Sobolevemb} together with H$\ddot{\text{o}}$lder's inequality we obtain 
$$
||u||_{L^r(\Omega)}\leq C_N
$$
for some constant $C_N$ independent of $u.$ Now arguing as in Lemma 2.1 of \cite{Aniso1} gives the existence of $u_n.$
\begin{itemize}
\item[(i)]
\begin{itemize}
\item[(1)] Let $\bar{p}<N$ and $g\in L^m(\Omega)$ for some $m>\frac{\bar{p}^{*}}{\bar{p}^{*}-\bar{p}}.$ Choosing $G_k(u_n)=(u_n-k)^{+}$ for $k>1$ as a test function in \eqref{approxexi} we get
$$
||G_k(u_n)||_{W_{0}^{1,p_i}(\Omega)}\leq e\Big(\int_{\Omega}g|G_k(u_n)|\,dx\Big)^\frac{1}{p_i}.
$$
Using Theorem \ref{Sobolevemb} with $r=\bar{p}^{*}$ and H$\ddot{\text{o}}$lder's inequality  we get
\begin{align*}
||G_k(u_n)||_{L^{\bar{p}^{*}}(\Omega)}\leq c\Big(\int_{A(k)}|g|^{\bar{p}^{*'}}\,dx\Big)^\frac{\bar{p}^{*}-1}{\bar{p}^{*}(\bar{p}-1)}.
\end{align*}
Now for $1<k<h$ denote by $A(h)=\{x\in\Omega:u(x)>h\}$, we get
\begin{align*}
&(h-k)^{p_i}|A(h)|^\frac{p_i}{\bar{p}^{*}}\\
&\leq\Big(\int_{A(k)}|G_k(u_n)|^{\bar{p}^{*}}\Big)^\frac{p_i}{\bar{p}^{*}}\\
&\leq c\Big(\int_{A(k)}|g|^{\bar{p}^{*'}}\,dx\Big)^\frac{p_i(\bar{p}^{*}-1)}{\bar{p}^{*}(\bar{p}-1)}.
\end{align*}
Now using H$\ddot{\text{o}}$lder's inequality with exponents $q=\frac{m}{\bar{p}^{*'}}$ and $q'=\frac{q}{q-1}$ we get
\begin{align*}
(h-k)^{p_i}|A(h)|^\frac{p_i}{\bar{p}^{*}}\leq c||g||^{\frac{p_i}{\bar{p}-1}}_{L^m(\Omega)}|A(k)|^{\frac{p_i(\bar{p}^{*}-1)(m-\bar{p}^{*'})}{\bar{p}^{*}m(\bar{p}-1)}}.
\end{align*}
Therefore we have
$$
|A(h)|\leq \frac{c||g||^{\frac{\bar{p}^{*}}{\bar{p}-1}}_{L^m(\Omega)}}{(h-k)^{\bar{p}^{*}}}|A(k)|^\beta,
$$
where $\beta=\frac{(\bar{p}^{*}-1)(m-\bar{p}^{*'})}{m(\bar{p}-1)}>1$ since $m>\frac{\bar{p}^{*}}{\bar{p}^{*}-\bar{p}}.$ By Stampacchia's result (\cite{Stamp}) we get $||u_n||_{L^\infty(\Omega)}\leq C$ where $C$ is independent of $n.$

\item[(2)] Choosing $G_k(u_n)=(u_n-k)^{+}$ as a test function in \eqref{approxexi} and using H${\ddot{\text{o}}}$lder's inequality we get
$$
||G_k(u_n)||_{W_{0}^{1,p_i}(\Omega)}\leq e||g||_{L^{r'}(A(k))}^\frac{1}{p_i-1}.
$$
Using H$\ddot{\text{o}}$lder's inequality with exponents $\frac{m}{r'}$ and $\frac{m}{m-r'}$ we get
$$
||G_k(u_n)||_{W_{0}^{1,p_i}(\Omega)}\leq c||g||_{L^m(\Omega)}^\frac{1}{p_i-1}|A(k)|^\frac{(m-r')}{mr'(p_i-1)}.
$$
Now for $1<k<h$ we have
\begin{align*}
&(h-k)^{p_i}|A(h)|^\frac{p_i}{r}\\
&\leq\Big(\int_{A(h)}(u-k)^r\,dx\Big)^\frac{p_i}{r}\\
&\leq\Big(\int_{A(k)}(u-k)^r\,dx\Big)^\frac{p_i}{r}\\
&\leq \sum_{i=1}^{N}\int_{\Omega}|\partial_{i}G_k(u_n)|^{p_i}\,dx\\
&\leq c||g||^{p_i'}_{L^m(\Omega)}|A(k)|^\frac{p_i(m-r')}{mr'(p_i-1)}.
\end{align*}
Therefore we have
$$
|A(h)|\leq c\frac{||g||^\frac{r}{p_i-1}|A(k)|^\gamma}{(h-k)^r},
$$
where $\gamma=\frac{r(m-r')}{mr'(p_i-1)}>1$ since $m>\frac{r}{r-p_{N}}.$ By Stampacchia's result (\cite{Stamp}) we get $||u_n||_{L^\infty(\Omega)}\leq C$ where $C$ is independent of $n.$
\end{itemize}
\item[(ii)] Let $u_n$ and $u_{n+1}$ satisfies the equations \eqref{approxexi}. Then for every $\phi\in W_{0}^{1,p_i}(\Omega)$
\begin{equation}\label{mono1}
\sum_{i=1}^{N}\int_{\Omega}|(u_n)_i|^{p_i-2}(u_n)_i \phi_i\,dx=\int_{\Omega}g_{n}e^\frac{1}{(u_n+\frac{1}{n})}\phi\,dx
\end{equation}
and
\begin{equation}\label{mono2}
\sum_{i=1}^{N}\int_{\Omega}|(u_{n+1})_i|^{p_i-2}(u_{n+1})_i \phi_i\,dx=\int_{\Omega}g_{n+1}e^\frac{1}{(u_{n+1}+\frac{1}{n+1})}\phi\,dx.
\end{equation}

Choosing $\phi=(u_n-u_{n+1})^{+}$ as a test function and subtracting \eqref{mono1} and \eqref{mono2} we have
\begin{align*}
&\sum_{i=1}^{N}\int_{\Omega}\Big(|(u_n)_{i}|^{p_i-2}(u_{n})_{i}-|(u_{n+1})_{i}|^{p_i-2}(u_{n+1})_i\Big)(u_n-u_{n+1})_{i}^{+}\,dx\\
&\leq \int_{\Omega}{g_{n+1}(x)}\left\{e^\frac{1}{(u_n+\frac{1}{n})}-e^\frac{1}{(u_{n+1}+\frac{1}{n+1})}\right\}(u_n-u_{n+1})_{i}^{+}\,dx\leq 0.
\end{align*}
Using the algebraic inequality (Lemma A.0.5 of \cite{Pe}) we get for any $p_i\geq 2$
$$||(u_n-u_{n+1})^{+}||_{W_{0}^{1,p_i}(\Omega)}=0.$$ Therefore (i) holds. The uniqueness follows similarly as in the monotonicity. 

\item[(iii)] Observe that $u_1\in L^\infty(\Omega)$ by using (i). Hence
$$
\sum_{i=1}^{N}\int_{\Omega}|(u_1)_i|^{p_i-2}(u_1)_i \phi_i\,dx=g_{1}e^\frac{1}{(u_{1}+1)}\geq g_1 e^\frac{1}{||u_1||_{\infty}+1}.
$$
Since $g$ is nonnegative and not identically zero, by the strong maximum principle (Theorem 3.18 of \cite{castrothesis}) we get the property (iii).
\end{itemize}
\end{proof}
\textbf{Proof of Theorem \ref{Mainexi}}
Let $\bar{p}<N$ such that $\bar{p}^{*}\geq p_N$ and $\Omega=\cup_{k}\Omega_k$ where $\Omega_k\subset\subset\Omega_{k+1}$ for each $k$. Let $\gamma_k=\text{inf}_{\Omega_k}\,u_n>0.$ Choosing $\phi=(u_n-\gamma_1)^{+}$ as a test function in \eqref{approxexi}, using Lemma \ref{approxexi} and Theorem \ref{Sobolevemb} we get
\begin{align*}
&\sum_{i=1}^{N}\int_{\{u_n>\gamma_1\}}|(u_n)_i|^{p_i}\,dx\\
&=\int_{\{u_n>\gamma_1\}}g_ne^\frac{1}{(u_n+\frac{1}{n})}(u_n-\gamma_1)^{+}\,dx\\
&\leq c||g||_{L^m(\Omega)}||(u_{n}-\gamma_{1})^{+}||_{W_{0}^{1,p_i}(\Omega)}
\end{align*}
where $c$ is a constant independent of $n$.
Using Lemma \ref{approxexi} and the fact
$$
||u_n||_{W^{1,p_i}(\Omega_1)}\leq ||u_n||_{W^{1,p_i}(\{u_n>\gamma_1\})}
$$
we get the sequence $\{u_n\}$ is uniformly bounded in $W^{1,p_i}(\Omega_1)$ and as a consequence of Theorem \ref{Sobolevemb} it has a subsequence $\{u^{1}_{n_k}\}$ converges weakly in $W^{1,p_i}(\Omega_1)$ and strongly in $L^{p_i}(\Omega_1)$ and almost everywhere in $\Omega_1$ to $u_{\Omega_1}\in W^{1,p_i}(\Omega_1)$, say.

Proceeding in the same way for any $k$, we obtain a subsequence $\{u_{n_l}^{k}\}$ of $\{u_n\}$ such that $u_{n_l}^{k}$ converges weakly in $W^{1,p_i}(\Omega_k)$, strongly in $L^{p_i}(\Omega_k)$ and almost everywhere to $u_{\Omega_k}\in W^{1,p_i}(\Omega_k).$ We may assume $u_{n_l}^{k+1}$ is a subsequence of $u_{n_l}^{k}$ for every $k$, and that $n_k^{k}\to\infty$ as $k\to\infty.$ Therefore $u_{\Omega_{k+1}}=u_{\Omega_k}$ on $\Omega_k$. Define $u=u_{\Omega_1}$ and $u=u_{\Omega_{k+1}}$ on $\Omega_{k+1}\setminus \Omega_k$ for each $k$. Therefore by our construction the diagonal subsequence $\{u_{n_k}\}:=\{u_{n_k}^{k}\}$ converges weakly to $u$ in $W^{1,p_i}_{loc}(\Omega_k)$, strongly in $L^{p_i}(\Omega_k)$ and almost everywhere in $\Omega.$ Now we claim that $\{u_{n_k}\}$ converges strongly to $u$ in $W^{1,p_i}_{loc}(\Omega_k).$ Let $\Omega'\subset\subset\Omega.$ Let $\phi\in C_c^{\infty}(\Omega)$ such that $0\leq\phi\leq 1$ in $\Omega$, $\phi=1$ on $\Omega'$ and let $k_1\geq 1$ such that $\text{supp}\,\phi\subset\Omega_{k_1}.$ For every $k,m\geq 1$ we have
\begin{align*}
&\sum_{i=1}^{N}\int_{\Omega'}\Big(|(u_{n_k})_{i}|^{p_i-2}(u_{n_k})_i-|(u_{n_m})_{i}|^{p_i-2}(u_{n_m})_i\Big)(u_{n_k}-u_{n_m})_{i}\,dx\\
&\leq \sum_{i=1}^{N}\int_{\Omega}\Big(|(u_{n_k})_{i}|^{p_i-2}(u_{n_k})_i-|(u_{n_m})_{i}|^{p_i-2}(u_{n_m})_{i}\Big)\Big(\phi(u_{n_k}-u_{n_m})\Big)_{i}\,dx-\\
&\sum_{i=1}^{N}\int_{\Omega_{k_1}}\Big\{\Big(|(u_{n_k})_{i}|^{p_i-2}(u_{n_k})_i-|(u_{n_m})_{i}|^{p_i-2}(u_{n_m})_i\Big).\phi_{i}\Big\}(u_{n_k}-u_{n_m})\,dx\\
&:=A-B.
\end{align*}
Now the fact that $u_{n_k}$ is uniformly bounded in $W^{1,p_i}(\Omega_{k_1})$ and converges strongly in $L^{p_i}(\Omega_{k_1})$ implies $B\to 0$ as $k,m\to\infty.$ Choosing $\psi=\phi(u_{n_k}-u_{n_m})$ and either $n=n_k$ or $n=n_m$ we get for $l=k,m$
\begin{align*}
&\Big|(u_{n_l})_{i}|^{p_i-2}(u_{n_l})_i\Big(\phi(u_{n_k}-u_{n_m})\Big)_{i}\,dx\Big|\\
&\leq\int_{\Omega_{k_1}}g_n(x)e^\frac{1}{(u_{n_l}+\frac{1}{n_l})}|u_{n_k}-u_{n_m}|\,dx.
\end{align*}
Now Lemma \ref{approxexi}, $g\in L^{m}(\Omega)$ and the strong convergence of $u_{n_k}$ gives $A\to 0$ as $k,m\to\infty.$ Now the algebraic inequality (Lemma A.0.5 of \cite{Pe}) gives 
$$
\sum_{i=1}^{N}\int_{\Omega'}|(u_{n_k})_i-(u_{n_m})_i|^{p_i}\,dx\to 0
$$
as $k,m\to\infty.$ Therefore for any $\phi\in C_c^{1}(\Omega)$ we have
$$
\sum_{i=1}^{N}\int_{\Omega}|(u_{n_k})_{i}|^{p_i-2}(u_{n_k})_{i}\phi_i\,dx=\sum_{i=1}^{N}\int_{\Omega}|u_{i}|^{p_i-2}u_{i}\phi_i\,dx.
$$
Lemma \ref{approxexi} and the fact $u_{n_{k}}\geq c_{\text{supp}\,\phi}>0$ gives
$$
\Big|\int_{\Omega}g_{n_k}(x)e^\frac{1}{(u_{n_k}+\frac{1}{n_k})}\phi\,dx\Big|\leq e^{\frac{1}{c_{\text{supp}\,\phi}}}||\phi||_{L^\infty(\Omega)}||g||_{L^1(\Omega)}.
$$
By Lebesgue dominated theorem we obtain
$$
\int_{\Omega}g_{n_k}(x)e^\frac{1}{(u_{n_k}+\frac{1}{n_k})}\phi\,dx=\int_{\Omega}g(x)e^\frac{1}{u}\phi\,dx.
$$
Hence $u\in W^{1,p_i}_{loc}(\Omega)$ is a weak solution of the problem \eqref{weak solution}. Now observe that $(u_{n_k}-\epsilon)^{+}$ in bounded in $W_{0}^{1,p_i}(\Omega)$ and it has a subsequence converges to $v$ weakly in $W_{0}^{1,p_i}(\Omega)$. Since $u_{n_{k}}$ converges almost everywhere to $u$, we have $v=(u-\epsilon)^{+}\in W_{0}^{1,p_i}(\Omega).$ The case $\bar{p}\geq N$ follows similarly using Theorem \ref{Sobolevemb}.
\section{Proof of nonexistence results}
To prove our main results we establish the following a priori estimate on the stable solution to the problem $\eqref{main equation}.$
\subsection{A priori estimate}
\begin{lemma}\label{a-priori}
Let $u\in W_{loc}^{1,p_i}(\Omega)$ be a positive stable solution to either of the equation $(\ref{main equation})_s$ or $(\ref{main equation})_e$ and $\alpha>p_N-1$ be fixed. Then for every $\epsilon\in(0,\alpha)$, there exists a positive constant $C=C_{\epsilon}(p_1,p_2,\cdots,p_N,q,\alpha)$ such that for any nonnegative $\psi\in C_c^{1}(\Omega)$, one has
\begin{equation}\label{apriorieqn}
\begin{aligned}
&\int_{\Omega}g(x) uf'(u)b_{k}(u)\psi^q\,dx\\
&\leq C\sum_{i=1}^{N}\int_{\Omega}u^{p_i-\alpha-1}|\psi_i|^{p_i}\psi^{q-p_i}\,dx\\
&\quad -\frac{(\alpha-1)^2(N(q-1)+\epsilon)}{4\alpha(1-\epsilon)}\int_{\Omega} g(x) f(u)b_{k}(u)\psi^q\,dx.
\end{aligned}
\end{equation}
\end{lemma}

As a corollary of Lemma \ref{a-priori} we obtain the following Caccioppoli type estimates.

\begin{corollary}\label{cor1}
Let $u\in W_{loc}^{1,p_i}(\Omega)$ be a positive stable solution to the problem $\eqref{main equation}_s$. Then the following holds:
\begin{enumerate}
\item[(1)] Assume that $0<u\leq 1$ a.e. in $\Omega$ and $1\leq\delta<\gamma$ be such that $\delta\in A\cap I$. Then for any $\beta\in(l_1,l_2),$ there exists a constant $C=C(p_1,p_2,\cdots,p_N,q,N,\beta)$ such that for every $\psi\in C_c^{1}(\Omega)$ with $0\leq\psi\leq 1$ in $\Omega$, we have
\begin{equation}\label{cac1}
\int_{\Omega}g(x)(\frac{\psi}{u})^{2\beta+\delta+q-1}\,dx\leq C\sum_{i=1}^{N}\int_{\Omega}|\psi_{i}|^{p_i\theta_i^{'}}\,dx
\end{equation}  where
\[
\theta_i=\frac{2\beta+\delta+q-1}{2\beta+q-p_i}, \quad
\theta_i^{'}=\frac{2\beta+\delta+q-1}{\delta+p_i-1}.
\]

\item[(2)] Assume that $u\geq 1$ a.e. in $\Omega$ and $0<\delta<\gamma$ be such that $\delta\in A$ and $\gamma\in I\cap[1,\infty).$ Then for any $\beta\in(l_1,l_2),$ there exists a constant $C=C(p_1,p_2,\cdots,p_N,q,N,\beta)$ such that for every $\psi\in C_c^{1}(\Omega)$ with $0\leq\psi\leq 1$ in $\Omega$, we have
\begin{equation}\label{caci1}
\int_{\Omega}g(x)(\frac{\psi}{u})^{2\beta+\gamma+q-1}\,dx\leq C\sum_{i=1}^{N}\int_{\Omega}|\psi_{i}|^{p_i\zeta_i^{'}}\,dx
\end{equation}  where
\[
\zeta_i=\frac{2\beta+\gamma+q-1}{2\beta+q-p_i}, \quad
\zeta_i^{'}=\frac{2\beta+\gamma+q-1}{\gamma+p_i-1}.
\]
\item[(3)] Assume that $u>0$ a.e. in $\Omega$ and $1\leq\delta=\gamma\in A\cap I.$ Then for any $\beta\in(l_1,l_2),$ there exists a constant $C=C(p_1,p_2,\cdots,p_N,q,N,\beta)$ such that for every $\psi\in C_c^{1}(\Omega)$ with $0\leq\psi\leq 1$ in $\Omega$, we have
\begin{equation}\label{cacio1}
\int_{\Omega}g(x)(\frac{\psi}{u})^{2\beta+\delta+q-1}\,dx\leq C\sum_{i=1}^{N}\int_{\Omega}|\psi_{i}|^{p_i\theta_i^{'}}\,dx
\end{equation}  where
\[
\theta_i=\frac{2\beta+\delta+q-1}{2\beta+q-p_i}, \quad
\theta_i^{'}=\frac{2\beta+\delta+q-1}{\delta+p_i-1}.
\]
\end{enumerate}
\end{corollary}

\begin{corollary}\label{cor2}
Let $u\in W_{loc}^{1,p_i}(\Omega)$ be a positive stable solution to the problem $\eqref{main equation}_e$ such that $0<u\leq M$ a.e. in $\Omega$ for some positive constant $M.$ Then for any $\beta\in(l_1,l_3)$ there exists a constant $C=C(p_1,p_2,\cdots,p_N,q,N,\beta)$ such that for every $\psi\in C_c^{1}(\Omega)$ with $0\leq\psi\leq 1$ in $\Omega$, we have
\begin{equation}\label{cac2}
\int_{\Omega}g(x)(\frac{\psi}{u})^{2\beta+q}\,dx\leq C\sum_{i=1}^{N}\int_{\Omega}|\psi_{i}|^{2\beta+q}\,dx.
\end{equation}
\end{corollary}

\begin{proof}[Proof of Lemma $\ref{a-priori}$] Let $u\in W_{loc}^{1,p_i}(\Omega)$ be a positive stable solution to the equation $(\ref{main equation})$ and $\psi\in C_c^{1}(\Omega)$ be nonnegative in $\Omega.$ Then $u$ satisfies both the equations $(\ref{weak solution})$ and $(\ref{stable solution}).$ We prove the lemma into the following two steps.
\smallskip

\noindent\textbf{Step 1.}
Choosing $\phi=b_k(u)\psi^q$ as a test function in $(\ref{weak solution}),$ we have
\begin{equation}\label{e11}
\begin{aligned}
&\sum_{i=1}^{N}\int_{\Omega} |b_{k}'(u)||u_i|^{p_i}\psi^q\,dx\\
&\leq q\sum_{i=1}^{N}\int_{\Omega} \psi^{q-1}b_{k}(u)|u_i|^{p_i-2}u_{i}\psi_{i}\,dx-\int_{\Omega}g(x)f(u)b_k(u)\psi^q\,dx.
\end{aligned}
\end{equation}
Using Young's inequality with $\epsilon\in(0,1)$, we obtain
\begin{align*}
&q\sum_{i=1}^{N}\int_{\Omega} \psi^{q-1}b_{k}(u)|u_i|^{p_i-2}u_i\psi_{i}\,dx\\
&\leq\epsilon\sum_{i=1}^{N}\int_{\Omega}|b_{k}'(u)||u_i|^{p_i}\psi^q\,dx+C\sum_{i=1}^{N}\int_{\Omega}{b_k(u)}^{p_i}|b_{k}'(u)|^{1-p_i}|\psi_{i}|^{p_i}\psi^{q-p_i}\,dx,
\end{align*}
for some positive constant depending $C=C_{\epsilon}(p_1,p_2,\cdots,p_N,q).$

Therefore for $\epsilon\in(0,1)$, we obtain
\begin{equation}\label{westimate}
\begin{aligned}
&(1-\epsilon)\sum_{i=1}^{N}|b_{k}'(u)||u_i|^{p_i}\psi^{q}\,dx\\
&\leq C\sum_{i=1}^{N}\int_{\Omega}{b_{k}(u)}^{p_i}|b_{k}'(u)|^{1-p_i}|\psi_i|^{p_i}\psi^{q-p_i}\,dx-\int_{\Omega} g(x) f(u)b_{k}(u)\psi^q\,dx.
\end{aligned}
\end{equation}

\textbf{Step 2.}
Choosing $\phi=a_{k}(u)\psi^\frac{q}{2}$ in the inequality $(\ref{stable solution})$, we obtain
\begin{equation}\label{sestimate}
\begin{gathered}
\int_{\Omega} g(x) f'(u)a_{k}(u)^2 \psi^q\,dx
\leq \sum_{i=1}^{N}(p_i-1)\left(X_i+\frac{q^2}{4} Y_i+q Z_i\right),
\end{gathered}
\end{equation}
where
$$
X_i=\int_{\Omega}|a_{k}'(u)|^2|u_i|^{p_i}\psi^q\,dx,\, Y_i=\int_{\Omega}\psi^{q-2}a_{k}(u)^2|u_i|^{p_i-2}|\psi_i|^2\,dx,
$$ and
$$
Z_i=\int_{\Omega}|a_{k}'(u)|a_{k}(u)\psi^{q-1}|u_i|^{p_i-1}|\psi_{i}|\,dx.
$$
Using c. noting that 
$$
X_i=\frac{(\alpha-1)^2}{4\alpha}\int_{\Omega}|b_{k}'(u)||u_i|^{p_i}\psi^q\,dx,
$$
from the estimate $(\ref{westimate})$, we obtain 
\begin{equation*}
\begin{aligned}
&\sum_{i=1}^{N}X_i=\frac{(\alpha-1)^2}{4\alpha}\sum_{i=1}^{N}\int_{\Omega}|b_{k}'(u)||u_i|^{p_i}\psi^q\,dx\\
&\leq \frac{(\alpha-1)^2}{4\alpha(1-\epsilon)}\left\{C\sum_{i=1}^{N}\int_{\Omega}{b_{k}(u)}^{p_i}|b_{k}'(u)|^{1-p_i}|\psi_i|^{p_i}\psi^{q-p_i}\,dx-\int_{\Omega} g(x) f(u)b_{k}(u)\psi^q\,dx\right\}.
\end{aligned}
\end{equation*}
Moreover, using Young's inequality we have
\begin{align*}
&(p_i-1)\frac{q^2}{4}Y_i\\
&=(p_i-1)\frac{q^2}{4}\int_{\Omega}\psi^{q-2}a_{k}(u)^2|u_i|^{p_i-2}|\psi_i|^2\,dx\\
&=(p_i-1)\frac{q^2}{4}\int_{\Omega}\left(|u_i|^{p_i-2}|a_{k}'(u)|^\frac{2(p_i-2)}{p_i}\psi^\frac{q(p_i-2)}{p_i}\right)\left(a_{k}(u)^2|a_{k}'(u)|^\frac{2(2-p_i)}{p_i}|\psi_i|^2\psi^\frac{2(q-p_i)}{p_i}\right)\,dx\\
&\leq \frac{\epsilon}{2N}X_i+\frac{C}{2}\int_{\Omega}a_k(u)^{p_i}|a_{k}'(u)|^{2-p_i}|\psi_i|^{p_i}\psi^{q-p_i}\,dx,
\end{align*}
and
\begin{align*}
&(p_i-1)qZ_i\\
&=(p_i-1)q\int_{\Omega}|a_{k}'(u)|a_{k}(u)\psi^{q-1}|u_i|^{p_i-1}|\psi_{i}|\,dx\\
&=(p_i-1)q\int_{\Omega}\left(|u_i|^{p_i-1}|a_{k}'(u)|^\frac{2}{p_{i}'}\psi^\frac{q}{p_{i}'}\right)\left(a_k(u)|a_{k}'(u)|^{\frac{2-p_i}{p_i}}|\psi|^{p_i}\psi^{q-p_i}\right)\,dx\\
&\leq \frac{\epsilon}{2N}X_i+\frac{C}{2}\int_{\Omega}a_{k}(u)^{p_i}|a_{k}'(u)|^{2-p_i}|\psi_i|^{p_i}\psi^{q-p_i}\,dx
\end{align*}
for some positive constant $C=C_{\epsilon}(p_1,p_2,\cdots,p_N,q,N).$

Using the above estimates in $(\ref{sestimate})$ together with a. and b. we obtain
\begin{gather*}
\int_{\Omega}g(x) uf'(u)b_{k}(u)\psi^q\,dx\leq \int_{\Omega}g(x) f'(u)a_{k}(u)^{2}\psi^q\,dx\\
\leq \sum_{i=1}^{N}\Big(p_i-1+\frac{\epsilon}{N}\Big)X_i+C\sum_{i=1}^{N}\int_{\Omega}a_{k}(u)^{p_i}|a_{k}'(u)|^{2-p_i}|\psi_i|^{p_i}\psi^{q-p_i}\,dx\\
\leq\Big(p_1-1+\frac{\epsilon}{N}\Big)\sum_{i=1}^{N}X_i+\Big(p_2-1+\frac{\epsilon}{N}\Big)\sum_{i=1}^{N}X_i+\cdots+\Big(p_N-1+\frac{\epsilon}{N}\Big)\sum_{i=1}^{N}X_i \\+ C\sum_{i=1}^{N}\int_{\Omega}a_{k}(u)^{p_i}|a_{k}'(u)|^{2-p_i}|\psi_i|^{p_i}\psi^{q-p_i}\,dx\\ 
=\Big(N(q-1)+\epsilon\Big)\sum_{i=1}^{N}X_i+C\sum_{i=1}^{N}\int_{\Omega}a_{k}(u)^{p_i}|a_{k}'(u)|^{2-p_i}|\psi_i|^{p_i}\psi^{q-p_i}\,dx\\
\leq \frac{(\alpha-1)^2\big(N(q-1)+\epsilon\big)}{4\alpha(1-\epsilon)}\Big\{C\sum_{i=1}^{N}\int_{\Omega}{b_{k}(u)}^{p_i}|b_{k}'(u)|^{1-p_i}|\psi_i|^{p_i}\psi^{q-p_i}\,dx\\-\int_{\Omega}g(x) f(u)b_{k}(u)\psi^q\,dx\Big\} + C\sum_{i=1}^{N}\int_{\Omega}a_{k}(u)^{p_i}|a_{k}'(u)|^{2-p_i}|\psi_i|^{p_i}\psi^{q-p_i}\,dx\\
\leq C\sum_{i=1}^{N}\int_{\Omega}\Big\{b_{k}(u)^{p_i}|b_{k}'(u)|^{1-p_i}+a_{k}(u)^{p_i}|a_{k}'(u)|^{2-p_i}\Big\}|\psi_i|^{p_i}\psi^{q-p_i}\,dx\\ - \frac{(\alpha-1)^2(N(q-1)+\epsilon)}{4\alpha(1-\epsilon)}\int_{\Omega} g(x)f(u)b_{k}(u)\psi^q\,dx\\
\leq C\sum_{i=1}^{N}\int_{\Omega}u^{p_i-\alpha-1}|\psi_i|^{p_i}\psi^{q-p_i}\,dx\\-\frac{(\alpha-1)^2(N(q-1)+\epsilon)}{4\alpha(1-\epsilon)}\int_{\Omega} g(x) f(u)b_{k}(u)\psi^q\,dx
\end{gather*}
for some positive constant $C=C_{\epsilon}(p_1,\cdots,p_N,q,N,\alpha).$
\end{proof}

\begin{proof}[Proof of Corollary \ref{cor1}] Let $u\in W_{loc}^{1,p_i}(\Omega)$ be a positive stable solution to the problem $\eqref{main equation}_s$. Observe that the fact $\beta>l_1$ implies $\alpha=2\beta+q-1>p_N-1$. Then by Lemma \ref{a-priori}, using the fact $0<\delta\leq\gamma$ and $f(u)=-u^{-\delta}-u^{-\gamma}$ in the inequality \eqref{apriorieqn}, for some $C=C_{\epsilon}(p_1,\cdots,p_N,q,N,\alpha)$ we obtain 
\begin{align*}
\alpha_{\epsilon}\int_{\Omega}g(x)(u^{-\delta}+u^{-\gamma})b_{k}(u)\psi^q\,dx\leq C\sum_{i=1}^{N}\int_{\Omega}u^{p_i-\alpha-1}|\psi_i|^{p_i}\psi^{q-p_i}\,dx,
\end{align*}
where $\alpha_{\epsilon}=\delta-\frac{(\alpha-1)^2(N(q-1)+\epsilon)}{4\alpha(1-\epsilon)}.$ Observe that
$$
\lim_{\epsilon\to 0}\alpha_{\epsilon}=\delta-\frac{N(q-1)(\alpha-1)^2}{4\alpha}>0,\,\,\forall\,\,\beta\in(l_1,l_2).
$$
Hence we can fix $\beta\in(l_1,l_2)$ and choose $\epsilon\in(0,1)$ such that $\alpha_{\epsilon}>0.$ As a consequence we have
\begin{equation}\label{eqn1}
\int_{\Omega}g(x)(u^{-\delta}+u^{-\gamma})b_{k}(u)\psi^q\,dx\leq C\sum_{i=1}^{N}\int_{\Omega}u^{p_i-2\beta-q}|\psi_i|^{p_i}\psi^{q-p_i}\,dx
\end{equation}for some positive constant $C=C(p_1,\cdots,p_N,q,N,\alpha).$
\begin{enumerate}
\item[(1)] Since $\delta<\gamma$ and $0<u\leq 1$ a.e. in $\Omega$, for any $\beta\in(l_1,l_2)$ the inequality \eqref{eqn1} becomes
\begin{align*}
\int_{\Omega}g(x)u^{-\delta}b_k(u)\psi^q\,dx\leq C\sum_{i=1}^{N}\int_{\Omega}|u|^{p_i-2\beta-q}|\psi_i|^{p_i}\psi^{q-p_i}\,dx.
\end{align*}
By the monotone convergence theorem we obtain
\begin{align*}
\int_{\Omega}g(x)u^{-2\beta-\delta-q+1}\psi^q\,dx\leq C\sum_{i=1}^{N}\int_{\Omega}|u|^{p_i-2\beta-q}|\psi_i|^{p_i}\psi^{q-p_i}\,dx.
\end{align*}
Replacing $\psi$ by $\psi^\frac{2\beta+\delta+q-1}{q}$ and using the Young's inequality for $\epsilon\in (0,1)$ with the exponents $\theta_i=\frac{2\beta+\delta+q-1}{2\beta+q-p_i}$, $\theta_{i}{'}=\frac{2\beta+\delta+q-1}{\delta+p_i-1}$ in the above inequality we obtain
\begin{equation*}
\begin{aligned}
&\int_{\Omega} g(x)(\frac{\psi}{u})^{2\beta+\delta+q-1} \,dx\\
&\leq C\sum_{i=1}^{N}\int_{\Omega} u^{p_i-2\beta-q} \psi^{2\beta+\delta+q-p_i-1}|\psi_i|^{p_i} \,dx\\
&=C\sum_{i=1}^{N}\int_{\Omega} \big((\frac{\psi}{u})^{2\beta+q-p_i}\big)
\big(\psi^{\delta-1}|\psi_i|^{p_i}\big) \,dx\\
&\leq\epsilon\int_{\Omega} g(x)(\frac{\psi}{u})^{2\beta+\delta+q-1} \,dx+C\sum_{i=1}^{N}\int_{\Omega} g^{-\frac{\theta_{i}^{'}}{\theta_i}}\psi^{(\delta-1)\theta_i{'}}|\psi_i|^{p_{i}\theta_i^{'}} \,dx.
\end{aligned}
\end{equation*}

Using $\delta\geq 1$ and choosing $0\leq\psi\leq 1$ in $\Omega$ together with the fact $g\geq c$ we obtain
$$
\int_{\Omega} g(x)(\frac{\psi}{u})^{2\beta+\delta+q-1} \,dx\leq C\sum_{i=1}^{N}\int_{\Omega} |\psi_i|^{p_{i}\theta_i^{'}} \,dx,
$$
for some positive constant $C=C(p_1,\cdots,p_N,q,N,\beta).$
\item[(2)] Since $\delta<\gamma$ and $u\geq 1$ a.e. in $\Omega$, for any $\beta\in(l_1,l_2)$ the inequality \eqref{eqn1} becomes
\begin{align*}
\int_{\Omega}g(x)u^{-\gamma}b_k(u)\psi^q\,dx\leq C\sum_{i=1}^{N}\int_{\Omega}|u|^{p_i-2\beta-q}|\psi_i|^{p_i}\psi^{q-p_i}\,dx.
\end{align*}
By the monotone convergence theorem we obtain
\begin{align*}
\int_{\Omega}g(x)u^{-2\beta-\gamma-q+1}\psi^q\,dx\leq C\sum_{i=1}^{N}\int_{\Omega}|u|^{p_i-2\beta-q}|\psi_i|^{p_i}\psi^{q-p_i}\,dx.
\end{align*}
Replacing $\psi$ by $\psi^\frac{2\beta+\gamma+q-1}{q}$ and using the Young's inequality for $\epsilon\in (0,1)$ with the exponents $\zeta_i=\frac{2\beta+\gamma+q-1}{2\beta+q-p_i}$, $\zeta_{i}{'}=\frac{2\beta+\gamma+q-1}{\gamma+p_i-1}$ in the above inequality we obtain
\begin{equation*}
\begin{aligned}
&\int_{\Omega} g(x)(\frac{\psi}{u})^{2\beta+\gamma+q-1} \,dx\\
&\leq C\sum_{i=1}^{N}\int_{\Omega} u^{p_i-2\beta-q} \psi^{2\beta+\gamma+q-p_i-1}|\psi_i|^{p_i} \,dx\\
&=C\sum_{i=1}^{N}\int_{\Omega} \big((\frac{\psi}{u})^{2\beta+q-p_i}\big)
\big(\psi^{\gamma-1}|\psi_i|^{p_i}\big) \,dx\\
&\leq\epsilon\int_{\Omega} g(x)(\frac{\psi}{u})^{2\beta+\gamma+q-1} \,dx+C\sum_{i=1}^{N}\int_{\Omega} g^{-\frac{\zeta_{i}^{'}}{\zeta_i}}\psi^{(\gamma-1)\zeta_{i}^{'}}|\psi_i|^{p_{i}\zeta_{i}^{'}} \,dx.
\end{aligned}
\end{equation*}

Using $\gamma\geq 1$ and choosing $0\leq\psi\leq 1$ in $\Omega$ together with the fact $g\geq c$ we obtain
$$
\int_{\Omega} g(x)(\frac{\psi}{u})^{2\beta+\gamma+q-1} \,dx\leq C\sum_{i=1}^{N}\int_{\Omega} |\psi_i|^{p_{i}\zeta_i^{'}} \,dx,
$$
for some positive constant $C=C(p_1,\cdots,p_N,q,N,\beta).$
\item[(3)] Since $\delta=\gamma\geq 1$ and $u>0$ a.e. in $\Omega$, for any $\beta\in(l_1,l_2)$ the inequality \eqref{eqn1} becomes
\begin{align*}
\int_{\Omega}g(x)u^{-\delta}b_k(u)\psi^q\,dx\leq C\sum_{i=1}^{N}\int_{\Omega}|u|^{p_i-2\beta-q}|\psi_i|^{p_i}\psi^{q-p_i}\,dx.
\end{align*}
Now proceeding similarly as in Case (1) we obtain the required estimate.
\end{enumerate}
\end{proof}

\begin{proof}[Proof of Corollary \ref{cor2}]
Assume $M\in J$ and let $u\in W_{loc}^{1,p_i}(\Omega)$ be such that $0<u\leq M$ a.e. in $\Omega$ is a positive stable solution of the equation $(\ref{main equation})_e.$ Let $\beta\in(l_1,l_3)$ and define $\alpha=2\beta+q-1$. Observe that the fact $\beta>l_1$ implies $\alpha>p_N-1$. Therefore we can apply Lemma \ref{a-priori} to choose $f(u)=-e^\frac{1}{u}$ and use the assumption $0<u\leq M$ a.e. in $\Omega$ in the estimate (\ref{apriorieqn}) and obtain
\begin{align*}
\alpha_{\epsilon}\int_{\Omega}g(x) e^\frac{1}{u}b_{k}(u)\psi^q\,dx\leq C\sum_{i=1}^{N}\int_{\Omega}u^{p_i-\alpha-1}|\psi_i|^{p_i}\psi^{q-p_i}\,dx,
\end{align*}
for some positive constant $C=C_{\epsilon}(p_1,\cdots,p_N,q,N,\alpha)$ where $\alpha_{\epsilon}=\frac{1}{M}-\frac{(\alpha-1)^2\big(N(q-1)+\epsilon\big)}{4\alpha(1-\epsilon)}.$ Observe that
$$
\lim_{\epsilon\to 0}\alpha_{\epsilon}=\frac{1}{M}-\frac{N(q-1)(\alpha-1)^2}{4\alpha}>0,\,\,\forall\,\,\beta\in(l_1,l_3).
$$
Hence we can fix $\beta\in(l_1,l_3)$ and choose $\epsilon\in(0,1)$ such that $\alpha_{\epsilon}>0.$
Using $e^x>x$ for $x>0$, in the above estimate we obtain
$$
\int_{\Omega}g(x)\frac{1}{u}b_{k}(u)\psi^q\,dx\leq \int_{\Omega}g(x) e^\frac{1}{u}b_{k}(u)\psi^q\,dx\leq C\sum_{i=1}^{N}\int_{\Omega}u^{p_i-2\beta-q}|\psi_i|^{p_i}\psi^{q-p_i}\,dx,
$$
for some positive constant $C=C(\beta,p_1,\cdots,p_N,q,N).$ By the monotone convergence theorem we obtain
$$
\int_{\Omega}g(x)u^{-2\beta-q}\psi^q\,dx\leq C\sum_{i=1}^{N}\int_{\Omega}u^{p_i-2\beta-q}|\psi_i|^{p_i}\psi^{q-p_i}\,dx.
$$
Replacing $\psi$ by $\psi^\frac{2\beta+q}{q}$ and using the Young's inequality for $\epsilon\in (0,1)$ with exponents $\gamma_i=\frac{2\beta+q}{2\beta+q-p_i}$, $\gamma_{i}{'}=\frac{2\beta+q}{p_i}$ in the above inequality we obtain
\begin{equation*}
\begin{aligned}
&\int_{\Omega} g(x)(\frac{\psi}{u})^{2\beta+q} \,dx\\
&\leq C\sum_{i=1}^{N}\int_{\Omega} (\frac{\psi}{u})^{2\beta+q-p_i}|\psi_i|^{p_i}\,dx\\
&\leq\epsilon\int_{\Omega} g(x) (\frac{\psi}{u})^{2\beta+q} \,dx+C\sum_{i=1}^{N}\int_{\Omega} {g^{-\frac{\gamma_i{'}}{\gamma_i}}}|\psi_i|^{2\beta+q} \,dx.
\end{aligned}
\end{equation*}
Therefore, using the fact that $g\geq c$, we have
$$
\int_{\Omega} g(x) (\frac{\psi}{u})^{2\beta+q} \,dx
\leq C\sum_{i=1}^{N}\int_{\Omega}|\psi_i|^{2\beta+q} \,dx,
$$
for some positive constant $C=C(\beta,p_1,\cdots,p_N,q,N).$
\end{proof}
\subsection{Proof of the main results}
\begin{proof}[Proof of Theorem \ref{smain1}] Let $u\in W_{loc}^{1,p_i}(\Omega)$ be a stable solution of the equation $(\ref{main equation})_s$ such that $0<u\leq 1$ a.e. in $\Omega$. Then by Corollary \ref{cor1} we have
 \begin{equation*}
\int_{\Omega}g(x)(\frac{\psi}{u})^{2\beta+\delta+q-1}\,dx\leq C\sum_{i=1}^{N}\int_{\Omega}|\psi_{i}|^{p_i\theta_i^{'}}\,dx.
\end{equation*}
Choosing $\psi=\psi_R$ in the above inequality we obtain
\begin{equation}\label{main1ineq}
\int_{B_{R}(0)} g(x)(\frac{1}{u})^{2\beta+\delta+q-1} \,dx\leq C\sum_{i=1}^{N}R^{N-p_i{\theta_{i}'}},
\end{equation}
for some positive constant $C$ independent of $R$.
Observe that, $$\lim_{\beta\to l_2}(N-p_{i}\theta_{i}')=N-\frac{p_i(2l_{2}+\delta+q-1)}{\delta+p_{i}-1}<0$$ which follows from the assumption $\delta\in I$, since $\delta>\frac{N^2(q-1)(p_i-1)}{p_i(N(q-1)+4)-N^2(q-1)}$ for all $i=1,2,\cdots,N.$ As a consequence, we can choose $\beta\in(l_1,l_2)$, such that $N-p_{i}\theta_{i}'<0$ for all $i$. 
Therefore, letting $R\to\infty$ in $(\ref{main1ineq})$, we obtain
$$
\int_{\Omega} g(x)(\frac{1}{u})^{2\beta+\delta+q-1} \,dx=0,
$$
which is a contradiction.
\end{proof}

\begin{proof}[Proof of Theorem \ref{smain2}] Let $u\in W_{loc}^{1,p_i}(\Omega)$ be a stable solution of the equation $(\ref{main equation})_s$ such that $u\geq 1$ a.e. in $\Omega$. Then by Corollary \ref{cor1} we have
\begin{equation*}
\int_{\Omega}g(x)(\frac{\psi}{u})^{2\beta+\gamma+q-1}\,dx\leq C\sum_{i=1}^{N}\int_{\Omega}|\psi_{i}|^{p_i\zeta_i^{'}}\,dx.
\end{equation*}
Choosing $\psi=\psi_R$ in the above inequality we obtain
\begin{equation}\label{main1ineqnew}
\int_{B_{R}(0)} {g(x)}(\frac{1}{u})^{2\beta+\gamma+q-1} \,dx\leq C\sum_{i=1}^{N}R^{N-p_i{\zeta_{i}'}},
\end{equation}
for some positive constant $C$ independent of $R$.
Observe that, $$\lim_{\beta\to l_2}(N-p_{i}\zeta_{i}')=N-\frac{p_i(2l_{2}+\gamma+q-1)}{\delta+p_{i}-1}<0$$ which follows from the assumption $\gamma\in I$, since $\gamma>\frac{N^2(q-1)(p_i-1)}{p_i(N(q-1)+4)-N^2(q-1)}$ for all $i=1,2,\cdots,N.$ As a consequence, we can choose $\beta\in(l_1,l_2)$, such that $N-p_{i}\zeta_{i}'<0$ for all $i$. 

Therefore, letting $R\to\infty$ in $(\ref{main1ineqnew})$, we obtain
$$
\int_{\Omega} g(x)(\frac{1}{u})^{2\beta+\gamma+q-1} \,dx=0,
$$
which is a contradiction.
\end{proof}

\begin{proof}[Proof of Theorem \ref{smain3}] Let $u\in W_{loc}^{1,p_i}(\Omega)$ be a positive stable solution of the equation $(\ref{main equation})_s$. Then by Corollary \ref{cor1} we have
 \begin{equation*}
\int_{\Omega}g(x)(\frac{\psi}{u})^{2\beta+\delta+q-1}\,dx\leq C\sum_{i=1}^{N}\int_{\Omega}|\psi_{i}|^{p_i\theta_i^{'}}\,dx.
\end{equation*}
Now proceeding similarly as in Theorem \ref{smain1} we obtain
$$
\int_{\Omega} g(x)(\frac{1}{u})^{2\beta+\delta+q-1} \,dx=0,
$$
which is a contradiction.
\end{proof}

\begin{proof}[Proof of Theorem \ref{emain}]
Let $u\in W_{loc}^{1,p_i}(\Omega)$ be a stable solution to the problem $\eqref{main equation}_{e}$ such that $0<u\leq M$ a.e. in $\Omega.$ Then by Corollary \ref{cor2} we have
$$
\int_{\Omega}g(x)(\frac{\psi}{u})^{2\beta+q} \,dx\leq C\sum_{i=1}^{N}\int_{\Omega}|\psi_i|^{2\beta+q} \,dx.
$$
Choosing $\psi=\psi_R$ in the above inequality we obtain
\begin{equation}\label{e2}
\int_{B_{R}(0)}g(x)(\frac{1}{u})^{2\beta+q} \,dx\leq CR^{N-2\beta-q},
\end{equation}
where $C$ is a positive constant independent of $R.$
Observe that, since $M\in J$ we have $0<M<\frac{4}{N(N-1)(q-1)}$ which implies $N<2l_3+q$ and hence 
$$\lim_{\beta\to l_3}(N-2\beta-q)=N-2l_3-q<0.$$ As a consequence, we can choose $\beta\in(l_1,l_3)$ such that $N-2\beta-q<0.$ \\

Therefore, letting $R\to\infty$ in $(\ref{e2})$, we obtain
$$
\int_{\Omega} {g(x)}(\frac{1}{u})^{2\beta+q} \,dx=0,
$$
which is a contradiction. Hence the Theorem follows.
\end{proof}
\section*{Acknowledgements}
The author thanks Professor Adimurthi for some fruitful discussion on the topic. The author thanks TIFR CAM, Bangalore for the financial support.

\end{document}